\documentclass[12pt,leqno,fleqn]{amsart}  
\usepackage{amsmath,amstext,amsthm,amssymb,amsxtra}
\usepackage{txfonts} 
\usepackage[T1]{fontenc}
\usepackage{lmodern}

\usepackage{euler}   
\usepackage{amsfonts}
\usepackage{latexsym}

\usepackage{tikz}

\newcommand{\R}{\mathbb{R}}

\newcommand{\N}{\mathbb{N}}
\newcommand{\Z}{\mathbb{Z}}

\def\diam{\qopname\relax o{diam}}

\def\dist{\qopname\relax o{dist}}
\def\supp{\qopname\relax o{supp}}
\def\min{\qopname\relax o{min}}

\numberwithin{equation}{section}

\usepackage{mathtools}
\mathtoolsset{showonlyrefs,showmanualtags} 
\def\kint_#1{\mathchoice%
          {\mathop{\kern 0.2em\vrule width 0.6em height 0.69678ex depth -0.58065ex
                  \kern -0.8em \intop}\nolimits_{\kern -0.4em#1}}%
          {\mathop{\kern 0.1em\vrule width 0.5em height 0.69678ex depth -0.60387ex
                  \kern -0.6em \intop}\nolimits_{#1}}%
          {\mathop{\kern 0.1em\vrule width 0.5em height 0.69678ex depth -0.60387ex
                  \kern -0.6em \intop}\nolimits_{#1}}%
          {\mathop{\kern 0.1em\vrule width 0.5em height 0.69678ex depth -0.60387ex
                  \kern -0.6em \intop}\nolimits_{#1}}}

\usepackage{hyperref} 
\hypersetup{
    colorlinks=true,       
    linkcolor=blue,          
    citecolor=magenta,        
    filecolor=magenta,      
    urlcolor=cyan           
}

\allowdisplaybreaks
\swapnumbers

\theoremstyle{plain} 
\newtheorem{lemma}[equation]{Lemma} 
\newtheorem{proposition}[equation]{Proposition} 
\newtheorem{theorem}[equation]{Theorem} 
\newtheorem{corollary}[equation]{Corollary}

\theoremstyle{definition}
\newtheorem{definition}[equation]{Definition} 

\theoremstyle{remark}
\newtheorem{remark}[equation]{Remark}

\newtheorem*{ack}{Acknowledgment}
\numberwithin{equation}{subsection}

\title[Hardy inequalities II]{Hardy inequalities in Triebel--Lizorkin spaces II. \\Aikawa dimension} 


\author[L Ihnatsyeva]{Lizaveta Ihnatsyeva}   

\address[L.I.]{Department of Mathematics and Statistics, P.O. Box 35 (MaD), FIN-40014 University of Jyv\"askyl\"a, Finland}
\email{lizaveta.ihnatsyeva@aalto.fi}

\author[AV V\"ah\"akangas]{Antti V. V\"ah\"akangas}
\address[A.V.V.]{Department of Mathematics and Statistics,
P.O. Box 68, FI-00014 University of Helsinki, Finland} \email{antti.vahakangas@helsinki.fi}

\begin{document}

	\maketitle

\begin{abstract} 
We prove inequalities of Hardy type for functions in Triebel--Lizorkin spaces $F^s_{pq}(G)$ on a domain $G\subset \R^n$, whose boundary has the Aikawa dimension 
strictly less than
 $n-sp$. 
\end{abstract}

\section{Introduction}

In this paper, we study Hardy-type inequalities for functions
in Triebel--Lizorkin spaces $F^s_{pq}(G)$; see \cite{T99}
for the case of bounded smooth domains $G$.
We assume that the boundary $\partial G$ of a domain $G$ is
`thin', in the sense that
its Aikawa
dimension is strictly smaller than 
$n-sp$. The notion of the Aikawa dimension appears  
in connection with the quasiadditivity of Riesz capacity, \cite{aikawa};
subsequently, it has turned out to be useful in other questions in the theory of function spaces, 
see e.g.  \cite{FJ,sickel}.
In particular, it is known that, for every $f\in C^\infty_0(G)$, a  `classical' Hardy inequality
\begin{equation}\label{e.classical}
\int_G \frac{\lvert f(x)\rvert^p}{\dist(x,\partial G)^p}\,dx \le C\int_G \lvert \nabla f(x)\rvert^p\,dx\,,
\end{equation}
holds
if $1<p<n$ and
$\partial G$ is `thin', i.e., if
 $\mathrm{dim}_{\mathcal{A}}(\partial G)<n-p$.
Indeed, as it is observed in \cite{lehrback}, this result is implicitly contained in
\cite{KoskelaZhong}. 
On the other hand, it is well known that
inequality \eqref{e.classical} holds if $\R^n\setminus G$ is $(1,p)$ uniformly fat with $1<p\le n$, we refer to \cite{Lewis1988}. These two last results exhibit a dichotomy between `thin' and `fat' sets 
which manifests in Hardy-type inequalities. 

Though our main result is Theorem \ref{t.john_hardy_TL},
we also formulate and prove the following illustrative theorem under an additional assumption that $G$ is a John domain.

\begin{theorem}\label{t.john_hardy}
Let $n\ge 2$, $1<p<\infty$, and $0<s<\min\{1,n/p\}$.
Suppose that  $G$ is a John domain in $\R^n$ such that
$\mathrm{dim}_{\mathcal{A}}(\partial G)< n-sp$. Then, for every $f\in L^p(G)$,
\begin{equation}\label{e.Hardy_basic}
\begin{split}
\bigg(\int_{G} \frac{\lvert f(x)\rvert^p}{\mathrm{dist}(x,\partial{G})^{sp}}\,dx \bigg)^{1/p}
\le  C\bigg\{\lVert f\rVert_{L^p(G)} + \bigg(\int_{G} \int_{G} \frac{\lvert f(x)-f(y)\rvert^p}{\lvert x-y\rvert^{n+sp}}\,dy\,dx\bigg)^{1/p}\bigg\}\,,
\end{split}
\end{equation}
where a constant $C$  depends on parameters $n$, $s$, $p$, and $G$.
\end{theorem}

Recall that
bounded Lipschitz domains, and bounded domains with the interior cone condition,
are John domains. Also, the Koch snowflake $G$ is a John domain with
$\mathrm{dim}_{\mathcal{A}}(\partial G)=\log 4/ \log 3$.
For bounded Lipschitz domains $G$
in the `fat' case of $sp>1$,
inequality \eqref{e.Hardy_basic} holds for every $f\in C^\infty_0(G)$ without 
the $L^p$-term $\lVert f\rVert_{L^p(G)}$
 on the right hand side;
furthermore, the  $L^p$-term cannot
be omitted if $sp\le 1$, \cite{Dyda2004}. 
But in contrast with the `fat' case, if the boundary is `thin', the
vanishing boundary values play no role.

In \S \ref{s.application}, we apply Theorem \ref{t.john_hardy} to prove the
boundedness of the `zero extension operator' on fractional Sobolev spaces,
\begin{equation}\label{zero_ext_def}
E_0:W^{s,p}(G)\to W^{s,p}(\R^n)\,,\quad E_0 f(x)=\begin{cases} f(x)\,,\qquad &x\in G\,;\\ 0\,,\qquad &x\in \R^n\setminus G\,.\end{cases}
\end{equation}
Further extension results for fractional Sobolev spaces 
can be found in \cite{Z}.



\medskip

The following theorem is our main result; 
it is a generalisation of Theorem \ref{t.john_hardy}.
For the definition of Triebel--Lizorkin spaces on domains, we refer
to \S \ref{s.notation}.

\begin{theorem}\label{t.john_hardy_TL}
Suppose $G$ is a bounded or unbounded domain in $\R^n$, $n\ge 2$,
with a compact boundary.
Let  $1<p<\infty$ and $0<s<n/p$.
Then the following statements hold:
\begin{itemize}
\item[(A)]
Assume that $\mathrm{dim}_{\mathcal{A}}(\partial G)< n-sp$. Then,
for every $1\le q< \infty$ and $f\in F^{s}_{pq}(G)$,
\begin{equation}\label{hardy}
\bigg(\int_{G} \frac{\lvert f(x)\rvert^p}{\mathrm{dist}(x,\partial{G})^{sp}}\,dx\bigg)^{1/p}
\lesssim \lVert f\rVert_{F^s_{pq}(G)}\,.
\end{equation}
\item[(B)] Conversely, if 
$ \R^n\setminus G$ has zero Lebesgue measure and
inequality \eqref{hardy} holds for some
$1\le q\le \infty$ and all $f\in F^s_{pq}(G)$, then $\mathrm{dim}_{\mathcal{A}}(\partial G)< n-sp$.
\end{itemize}
\end{theorem}

For the class of domains whose boundary is compact and whose complement has zero Lebesgue measure, Theorem \ref{t.john_hardy_TL} gives further information.
First, in this class we obtain a characterization for the validity of Hardy inequality \eqref{hardy} in terms
of the Aikawa dimension of $\partial G$. In particular, the validity is seen to be independent of the microscopic parameter $q$.
Moreover, we recover a self-improving property 
of Hardy inequality \eqref{hardy} with respect to parameters $s$ and $p$;
this extends the results in \cite{KoskelaZhong} for the
classical Hardy inequality. 
Finally, since the notion of porosity is related
to the Aikawa dimension, see Remark \ref{porous_obs}, 
we immediately obtain a characterization for porosity, Corollary \ref{PorousityCharact}; analogous statements, showing  the sufficiency of porosity 
for the boundedness of pointwise multipliers, can be found in \cite{T03} and \cite[Proposition 3.19]{T4}. 


\begin{corollary}\label{PorousityCharact}
Suppose $G$ is
a  domain in $\R^n$ such that $\partial G$ is compact
and the Lebesgue measure of $\R^n\setminus G$ is zero. Then the following
statements are equivalent:
\begin{itemize}
\item[(1)]
There is $0<\epsilon<n$ such that,
if $1<p<\infty$ and $0<s<\epsilon/p$,  then
inequality \eqref{hardy} holds for every $1\le q<\infty$ and  
every $f\in F^{s}_{pq}(G)$.
\item[(2)] The boundary $\partial G$ is porous in $\R^n$.
\end{itemize}
\end{corollary}

Explicit boundary conditions for $f$ are not imposed in statement (A) of Theorem  \ref{t.john_hardy_TL}.
This is related to the fact that for a bounded domain $G$ with $\mathrm{dim}_{\mathcal{A}}(\partial G)< n-sp$,
\begin{equation}\label{indent}
 \overset\circ{F^{s}_{pq}}({G})= F^s_{pq}(G) = \widetilde{F}^s_{pq}(G)\,.
\end{equation}
In particular, $C^\infty_0(G)$ is dense in $F^s_{pq}(G)$.
These embeddings are available in the literature, some of them implicitly, \cite{caetano,FJ}.
The validity of Hardy inequality is related to the second identification
in \eqref{indent} or, more precisely,
to a question if
the pointwise multiplication $f\mapsto \chi_G f$ is bounded on $F^{s}_{pq}(\R^n)$.
As an application of our main result, Theorem \ref{t.john_hardy_TL}, we obtain a simple proof for certain known multiplier results in \cite[\S 13]{FJ}
for
 spaces $F^s_{pq}(\R^n)$, see 
also \S \ref{s.application}.


Theorem \ref{t.john_hardy_TL} applies in the case of `thin' sets.
We would also like to mention some of the known results
for  `fat' sets. In this case the
first identification
in \eqref{indent} often fails,
and it seems natural
to impose zero boundary conditions.
For an illustration, let us first
focus on the case of $G=\R^n\setminus S$, where
$S$ is a closed Ahlfors $d$-regular set in $\R^n$ with $n-1<d<n$
and $sp>n-d$. 
The $d$-regularity condition is given in terms of the  Hausdorff measure, that is, 
for every $x\in S$ and $0<r<1$,
\[
\mathcal{H}^d(B(x,r)\cap S)  \simeq r^d\,.\] 
For such sets, we have $\mathrm{dim}_{\mathcal{A}}(S) = d$, \cite[Lemma 2.1]{lehrback}.
In our previous work \cite{ihnatsyeva2} we have shown that
inequality \eqref{hardy} holds for functions $f$ in the subspace
\[
\{f\in F^{s}_{pq}(\R^n)\,:\, \mathrm{Tr}_{\partial G} f= 0 \}\,,
\]
where $\mathrm{Tr}_{\partial G} f$ is a trace of $f$ on $\partial G=S$.

Although the last result applies in a somewhat general setting,
it is probably not yet optimal. Indeed, it seems to be natural to replace the $d$-regularity condition  with 
uniform fatness, which is typically used
in connection with classical Hardy inequalities.
Moreover, it is known that a certain fractional Hardy-type inequality
holds,
if
$G$ is a bounded domain in $\R^n$ for which $\R^n\setminus G$ 
is $(s,p)$ locally uniformly fat, \cite[Theorem 1.3]{E-HSV}. Recall that a set $E\subset\R^n$ is $(s,p)$ locally uniformly fat,
if there are positive constants $r_0$ and $\lambda$ such that, for every $x\in E$ and $0<r<r_0$,
\[
R_{s,p}(B(x,r)\cap E)\ge \lambda r^{n-sp}\,.
\]
For example, $R_{s,p}(B(x,r))\simeq r^{n-sp}$
if $x\in \R^n$ and $r>0$.
 Here $R_{s,p}(\cdot)$ is
the $(s,p)$ outer Riesz capacity of a set in $\R^n$, we refer to \cite{Lewis1988}.


\begin{ack}
  AVV would like to thank Eero Saksman for  valuable
  discussions  during the research program Operator Related Function Theory and Time-Frequency Analysis at the Centre for Advanced Study at 
  the Norwegian Academy of Science and Letters in Oslo during 2012--2013. LI is supported by the Academy of Finland, grant 135561.
\end{ack}

%
%
%
%

\section{Notation and preliminaries}\label{s.notation}

\subsection*{Basic notation}
Throughout the paper, a cube means a closed cube $Q=Q(x,r)$ in $\R^n$ centred
at $x\in \R^n$ with side length $\ell(Q)=2r>0$, and with sides
parallel to the coordinate axes. For a cube $Q$ and for $\rho>0$, we write $
\rho Q$ for the dilated cube with side
length $\rho \ell(Q)$. By $\chi_E$ we denote the
characteristic function of a set $E$, the boundary of $E$ is written
as $\partial E$, and $\lvert E\rvert$ is the Lebesgue $n$-measure of a
measurable set $E$ in $\R^n$.
The integral average of $f\in L_\textup{loc}^1(\R^n)$ over a bounded set $E$ with positive measure is written as
$f_E$, that is,
\[
f_E = \fint_E f\, dx = \frac1{\rvert E \lvert}\int_E f\, dx\,.
\] 
Various constants whose value may change even within
a given line are denoted by $C$. 

The family of closed dyadic cubes is denoted by $\mathcal{D}$. Let also
$\mathcal{D}_j$ be the family of those dyadic cubes whose side length
is $2^{-j}$, $j\in\Z$. For a proper open set $G$ we fix its
Whitney decomposition $\mathcal{W}(G)\subset\mathcal{D}$, and write
$\mathcal{W}_j(G)=\mathcal{D}_j\cap \mathcal{W}(G)$. 
For a Whitney
cube $Q\in\mathcal{W}(G)$ we write $Q^*=\frac{9}{8} Q$. Such dilated cubes 
have a bounded overlap, and
they satisfy
\begin{equation}\label{dist_est}
  \frac{3}{4}\diam(Q)\le \dist(x,\partial G)\le 6\diam(Q),
\end{equation}
whenever $x\in Q^*$.  
For other properties of Whitney cubes we refer to \cite[VI.1]{Stein}.

%
%

\subsection*{Aikawa dimension and porous sets} \label{porous_sect}
We begin with the definition of the Aikawa dimension, \cite{lehrback}.
\begin{definition}\label{d.aikawa}
If $E\subset \R^n$ is a closed set with an empty interior, then
define $\mathcal{A}(E)$ to be the set  of all $0<s\le n$ with the following property. 
There
is a constant $C>0$, such that
\begin{equation}\label{e.assouad}
\int_{B(x,r)} \mathrm{dist}(y,E)^{s-n}\,dy\le C r^s
\end{equation}
for every $x\in E$ and all  $0<r<\infty$. The Aikawa dimension
of $E$ is $\mathrm{dim}_{\mathcal{A}}(E)=\inf \mathcal{A}(E)$.
\end{definition}

\begin{remark}\label{r.compact}
Let $E$ be a compact set in $\R^n$ with an empty interior.
Suppose that for $0<s\le n$
there are constants $\epsilon,C>0$ such that
\eqref{e.assouad} holds for all $x\in E$ and $0<r<\epsilon$.
By using compactness arguments it is straightforward to verify that  $s\in \mathcal{A}(E)$.
\end{remark}

\begin{definition}\label{porous}
A  set $S\subset\R^n$ is 
{\em porous} (or {\em $\kappa$-porous}) if for some $\kappa\ge 1$
the following statement is true:
For every cube $Q(x,r)$ with $x\in\R^n$ and $0<r\le 1$ there is $y\in Q(x,r)$ such that
$Q(y,r/\kappa)\cap S=\emptyset$.
\end{definition}

\begin{remark}\label{porous_obs}
A set $S$ 
is porous in $\R^n$ if, and only if, its so called Assouad dimension is strictly less 
than $n$, \cite{Luukkainen}. 
It has been recently shown in \cite{lehrbackII} that the dimensions of Assouad and Aikawa 
of a given set $S\subset \R^n$ 
coincide.
\end{remark}

Let us then recall a 
reverse H\"older type inequality involving porous sets, Theorem \ref{reverse}.
Similar techniques are applied in \cite[\S 13]{FJ}; we also refer to \cite{Bojarski} and \cite{iwaniec}.
For a set $S$ in $\R^n$ and a positive constant $\gamma>0$ we denote
\begin{equation}\label{c_definition}
\mathcal{C}_{S,\gamma} = \{Q\in\mathcal{D}\,:\,\gamma^{-1}\mathrm{dist}(x_Q,S)\le \ell(Q)\le 1\}\,.
\end{equation}
This is the family of dyadic cubes that are relatively close to the set.

\begin{theorem}\label{reverse}
Suppose that $S\subset\R^n$ is porous.
Let $p,q\in (1,\infty)$ and $\{a_Q\}_{Q\in\mathcal{C}_{S,\gamma}}$ be a sequence of non-negative scalars. 
Then
\begin{equation}\label{rev}
\bigg\|\sum_{Q\in\mathcal{C}_{S,\gamma}} \chi_Q a_Q\bigg\|_p\le C\bigg\|\bigg(\sum_{Q\in\mathcal{C}_{S,\gamma}}
(\chi_{Q} a_Q)^q\bigg)^{1/q}\bigg\|_p.
\end{equation}
Here a positive constant $C$ depends on $n$, $p$, $\gamma$ and the set $S$.
\end{theorem}

The proof of this theorem
is based on maximal-function techniques, we refer to \cite{ihnatsyeva}.

\subsection*{Function spaces}\label{spaces}

There are several equivalent characterizations for the fractional Sobolev spaces and their natural extensions, Triebel--Lizorkin spaces, see e.g. \cite{AH}, \cite{BesovIlinNikolski}, \cite{T89}, and \cite{T2}.
 In this
paper, we mostly use the definition based on the local polynomial approximation approach.

Let $f\in L^u_{\mathrm{loc}}(\R^n)$, $1\le u \le \infty$, and $k\in\N_0$.
Following \cite{Brudnyi74}, we define the {\em normalized local best approximation} of $f$ on a cube
$Q$ in $\R^n$ by
\begin{equation} 
\mathcal{E}_k(f,Q)_{L^u(\R^n)}:=\inf_{P\in \mathcal{P}_{k-1}}\bigg(\fint_{Q}|f(x)-P(x)|^u\;dx\bigg)^{1/u}.
\end{equation}
Here and below  $\mathcal{P}_{m}$, $m\geq 0$, denotes the space of polynomials in $\mathbb{R}^n$ of degree at most $m$. We also denote 
$\mathcal{P}_{-1}=\{0\}$.
Let $Q_1\subset Q_2$ be two cubes in $\R^n$. Then
\begin{equation}\label{eqMonotonyOfLocalApproxSset}
\mathcal{E}_k(f,Q_1)_{L^u(\R^n)}
\le \bigg(\frac{\ell(Q_2)}{\ell(Q_1)}\bigg)^{n/u}
\mathcal{E}_k(f,Q_2)_{L^u(\R^n)}.
\end{equation}
This property is referred as  {\em the monotonicity
of local approximation}.



The following definition of Triebel--Lizorkin spaces 
with positive smoothness
can be found in \cite{T89}.
Let $s>0$, $1\le p<\infty$, $1\le q\le\infty$, and $k$ be an integer such that $s<k$. 
For $f\in L^u_{\mathrm{loc}}(\R^n)$, $1\le u\le \min\{p,q\}$, set for 
all $x\in\R^n$,
\[
F(x):=\bigg(\int_0^1\bigg(\frac{\mathcal{E}_k(f,Q(x,t))_{L^u(\R^n)}}{t^{s}}\bigg)^{q}\;\frac{dt}{t}\bigg)^{1/q}\,,\qquad \text{if }q<\infty,
\]
and \[F(x):=\sup\{t^{-s}\mathcal{E}_k(f,Q(x,t))_{L^u(\R^n)}:\, 0<t\le 1\}\,,\qquad \text{if }q=\infty.
\] 
A function $f$ belongs to a
Triebel--Lizorkin space $F^s_{pq}(\R^n)$ if $f$ and $F$ are both in $L^p(\R^n)$, and the Triebel--Lizorkin norms
\[
\| f\|_{F^s_{pq}(\R^n)}:=\|f\|_{L^p(\R^n)}+\|F\|_{L^p(\R^n)}
\]
are equivalent if $s<k$ and $1\le u\le \min\{p,q\}$. In particular, if $q\ge p$, then
we can set $u=p$.


\subsection*{Function spaces on domains}
Let us recall the definition of the {\em fractional order Sobolev spaces} on a
domain $G\subset\R^n$.
Let
$W^{s,p}(G)$, for $0<s<1$ and  $1< p<\infty$, be the space of functions $f$ in
$L^p({G})$ with
$\lVert f\rVert_{W^{s,p}({G})}
:=\lVert f\rVert_{L^p({G})}+|f|_{W^{s,p}({G})}<\infty$,
where
\begin{equation}\label{d.frac_def}
|f|_{W^{s,p}({G})} := \bigg(\int_{G} \int_{G}
\frac{|f(x)-f(y)|^p}{|x-y|^{n+sp}}\,dy\,dx\bigg)^{1/p}.
\end{equation}

The Triebel--Lizorkin space $F^{s}_{pp}(\R^n)$ 
coincides with the Sobolev space $W^{s,p}(\R^n)$, \cite[pp. 6--7]{T2}.

Let us also recall some notation which is common in the literature on 
function spaces on domains, \cite{T2,T3}. 
Let $G$ be an open set in $\R^n$, $1\le p<\infty$,  $1\le q\leq\infty$,
and $s>0$. Then
\[
F^s_{pq}(G)=\{f\in L^p(G)\,:\, \text{there is a}\, g\in F^s_{pq}(\R^n)\,\,\text{with}\,g|_G=f\}
\]
\[
\Vert f\Vert_{F^s_{pq}(G)}=\inf\Vert g\Vert_{F^s_{pq}(\R^n)},
\]
where the infimum is taken over all $g\in F^s_{pq}(\R^n)$ such that $g|_G=f$ pointwise a.e. 
As usually, we also denote
\begin{equation}\label{definitionFtilde}
\widetilde{F}^s_{pq}(G)=\{f\in L^p(G):\, \text{there is a}\, g\in F^s_{pq}(\R^n)\,\,\text{with}\,g|_G=f \, \text{and} \, \supp g\subset\overline{G}\}
\end{equation}
\[
\Vert f\Vert_{\widetilde{F}^s_{pq}(G)}=\inf\Vert g\Vert_{F^s_{pq}(\R^n)},
\]
where the infimum is taken over all $g$ admitted in \eqref{definitionFtilde},

Finally, 
$\overset\circ{ F^{s}_{pq}}(G)$ is a completion of $C^\infty_0(G)$ in $F^s_{pq}(G)$.

\section{Proof of Theorem \ref{t.john_hardy}}\label{s.hardy_admissible}

We apply the ideas in \cite{h-smv} to prove a Hardy inequality for John domains.

\begin{definition}\label{sjohn}
  A bounded domain $G$ in $\R^n$, $n\ge 2$, is a {\em John domain}, if
  there exist a point $x_0\in G$ and a constant $\beta_G\ge 1$ such that every
  point $x$ in $G$ can be joined to $x_0$ by a rectifiable curve
  $\gamma:[0,\ell]\to G$ parametrized by its arc length for which
  $\gamma(0)=x$, $\gamma(\ell)=x_0$, $\ell \le \beta_G \diam(G)$, and for all $t\in [0,\ell]$,
\[
\dist(\gamma(t),\partial G)\ge t/\beta_G\,.
\]
The point $x_0$ a {\em John center}
of $G$, and the smallest constant $\beta_G\ge 1$ is the {\em
  John constant} of $G$.
\end{definition}

\subsection*{Chain decomposition}
Suppose that $G$ is a John domain and 
 $Q\in\mathcal{W}(G)$. Below,  we obtain a {\em chain} of cubes
\[
\mathcal{C}(Q)= (Q_0,\ldots, Q_m)\subset \mathcal{W}(G)\,,
\]
joining a fixed cube $Q_0$ to the given cube $Q=Q_m$, such that 
$Q_i\not=Q_j$ whenever $i\not=j$,
and there exists a positive constant
$C=C(n)$ for which
\begin{equation}\label{p.2}
\lvert Q_j^* \cap Q_{j-1}^*\rvert \ge C\max\{\lvert Q_j^*\rvert, \lvert Q_{j-1}^*\rvert\}\,,\qquad j\in \{1,\ldots,m\}\,.
\end{equation}
A family
$\{\mathcal{C}(Q):\ Q\in\mathcal{W}(G)\}$ is called
 a {\em chain decomposition} of $G$, and
the  {\em shadow} of a  Whitney cube $R\in\mathcal{W}(G)$ is 
\[
\mathcal{S}(R) = \{Q\in\mathcal{W}(G):\ R\in\mathcal{C}(Q)\}\,.
\]

The following proposition provides a chain decomposition
for the appropriate John domains.
\begin{proposition}(Chain decomposition)\label{p.admissible} Suppose
that $n\ge 2$,  $1<p<\infty$, and $0<s<n/p$.
Let $G$ be a  John domain in $\R^n$, with $\mathrm{dim}_{\mathcal{A}}(\partial G) < n-sp$.  Then
there exist constants $\sigma,\tau\in\N$ and a chain
decomposition $\{\mathcal{C}(Q):\ Q\in\mathcal{W}(G)\}$ of $G$ satisfying
the following conditions:
\begin{itemize}
\item[(1)] $\ell(Q)\le 2^{\tau}\ell(R)$ for each $Q\in\mathcal{W}(G)$ and $R\in\mathcal{C}(Q)$;
\item[(2)] $\sharp \{R\in\mathcal{W}_j(G):\ R\in\mathcal{C}(Q)\}\le
 2^ \tau$ for each $Q\in \mathcal{W}(G)$ and $j\in \Z$;
\item[(3)] The following inequality holds,
\begin{equation}\label{e.useful}
\sup_{j\in \Z}\sup_{R\in\mathcal{W}_j(G)}
\frac{1}{\lvert R\lvert^{1-sp/n}}
\sum_{k=j-\tau}^\infty \sum_{\substack{Q\in \mathcal{W}_k(G) \\ Q\in\mathcal{S}(R) }} \lvert Q\rvert^{1-sp/n}
(\tau+1+k-j)^{p}  < \sigma\,.
\end{equation}
\end{itemize}
The constants $\sigma$ and $\tau$ depend only
on $n$, $p$, $s$, $\partial G$, and the John constant $\beta_G$.
\end{proposition}

\begin{proof}
For the construction of chain decomposition and the verification of
conditions (1) and (2), we refer to \cite{h-smv}.
Therein one may also find a proof of the following useful fact.
There is
a constant $C=C(n,\beta_G)>0$ such that, for each $R\in\mathcal{W}(G)$,
\begin{equation}\label{e.standard}
\bigcup_{Q\in \mathcal{S}(R)} Q \subset B(y_R, C\ell(R)),
\end{equation}
where $y_R\in\partial G$ is any point satisfying $\lvert x_R-y_R\rvert
=\dist(x_R,\partial G)$.

It remains to check condition (3).  
Let us fix $\epsilon>0$, depending on the allowed parameters, such that $n-sp-\epsilon\in\mathcal{A}(\partial G)$. 
This can be done since, by the assumption, $\mathrm{dim}_{\mathcal{A}} (\partial G)<n-sp$.
Fix $j\in\Z$ and $R\in\mathcal{W}_j(G)$.
 Then, if
$k\ge j-\tau$ and $Q\in\mathcal{W}_k(G)$,
\begin{equation} \label{i.test}
\begin{split}
  \left(\frac{\ell (Q)}{\ell (R)}\right)^{\epsilon} (\tau + 1 + k - j)^p &= 2^{(\tau +1)\epsilon}
  2^{-(\tau + 1+k-j)\epsilon} (\tau + 1 + k-j)^p \le C2^{\tau \epsilon}\,,
  \end{split}
\end{equation}
where $C=C(\epsilon,p)>0$. By inequality~\eqref{i.test},
\begin{align*}
  \sum_{k=j-\tau}^\infty \sum_{\substack{Q\in \mathcal{W}_k(G)
      \\ Q\in\mathcal{S}(R) }}
  \left(\frac{\ell(Q)}{\ell(R)}\right)^{n-sp} (\tau+1+k-j)^{p} \le
  C2^{\tau \epsilon}
  \ell(R)^{-(n-sp-\epsilon)}\sum_{Q\in\mathcal{S}(R)}
  \ell(Q)^{n-sp-\epsilon}\,.
\end{align*}
On the other hand, by \eqref{dist_est}, \eqref{e.standard}, and \eqref{e.assouad}, we may
conclude that
\begin{equation*}
  \sum_{Q\in\mathcal{S}(R)} \ell(Q)^{n-sp-\epsilon} 
%
%
\le C\int_{B(y_R,C\ell(R))} \dist(x,\partial
G)^{(n-sp-\epsilon)-n}\,dx \leq C\ell(R)^{n-sp-\epsilon}\,,
\end{equation*}
and condition (3) follows.
\end{proof}

\subsection*{Hardy inequality for John domains} We are ready to verify one of our main results.


\begin{proof}[Proof of Theorem \ref{t.john_hardy}]
Let $\{\mathcal{C}(Q)\}$ be a chain decomposition given by Proposition \ref{p.admissible},
with a fixed cube $Q_0\in\mathcal{W}(G)$. 
Without loss of generality, we may  assume the normalisation $f_{Q_0^*} =0$. Indeed, if necessary,
we replace $f$ with $f-f_{Q_0^*}\chi_G$ in the proof below, and use the bound on the
Aikawa dimension to control the error term. 

Let us estimate
\begin{equation}\label{e.terms_basic}
\begin{split}
&\int_{G} \frac{\lvert f(x)\rvert^p}{\mathrm{dist}(x,\partial{G})^{sp}}\,dx  \\&\lesssim \sum_{Q\in \mathcal{W}(G)} \ell(Q)^{n-sp}\fint_{Q^*}
\lvert f(x)- f_{Q^*}\rvert^p\,dx + \sum_{Q\in\mathcal{W}(G)} \ell(Q)^{n-sp}\lvert f_{Q^*}\rvert^p\,.
\end{split}
\end{equation}
By H\"older's inequality, 
and the facts   $\lvert x-y\rvert\lesssim \ell(Q^*)$ for $(x,y)\in Q^*\times Q^*$
and $\sum_{Q\in\mathcal{W}(G)} \chi_{Q^*}\lesssim \chi_{G}$,
the first term on the right hand side is bounded by
\begin{equation}\label{e.first_basic}
\begin{split}
&\sum_{Q\in\mathcal{W}(G)} \ell(Q)^{n-sp}\fint_{Q^*} \fint_{Q^*} \lvert f(x)-f(y)\rvert^p\,dy\,dx
\\&\lesssim \sum_{Q\in\mathcal{W}(G)} \int_{Q^*} \int_{Q^*} \frac{\lvert f(x)-f(y)\rvert^p}{\lvert x-y\rvert^{n+sp}} \,dy\,dx\lesssim \int_G \int_G \frac{\lvert f(x)-f(y)\rvert^p}{\lvert x-y\rvert^{n+sp}}\,dy\,dx\,.
\end{split}
\end{equation}
In order to control the remaining term, let us first prove some auxiliary estimates.
For a Whitney cube $Q\in\mathcal{W}(G)$, consider its chain $\mathcal{C}(Q)$. By inequality \eqref{p.2},
if $j\in \{1,\ldots,m\}$,
\begin{align*}
\lvert f_{Q_j^*} - f_{Q_{j-1}^*}\rvert = \fint_{Q_j^*\cap Q_{j-1}^*} \lvert f_{Q_j^*}- f_{Q_{j-1}^*}\rvert \,dx
\lesssim \sum_{i=j-1}^j \fint_{Q_i^*} \lvert f(x)-f_{Q_i^*}\rvert\,dx
\end{align*}
By normalisation $f_{Q_0^*}=0$ and the property
that cubes in the chain $\mathcal{C}(Q)$ are distinct,
\begin{equation}\label{e.collect}
\begin{split}
\lvert f_{Q^*}\rvert &= \lvert f_{Q_m^*}-f_{Q_0^*}\rvert \\&\lesssim \sum_{j=1}^m \sum_{i=j-1}^j\fint_{Q_i^*} \lvert f(x)-f_{Q_i^*}\rvert\,dx
\lesssim \sum_{R\in\mathcal{C}(Q)} \fint_{R^*} \lvert f(x)-f_{R^*}\rvert\,dx\,.
\end{split}
\end{equation}
We are ready to estimate the second term in the right hand side of \eqref{e.terms_basic}.
First we will use inequality \eqref{e.collect} and property (1) of the chain $\mathcal{C}(Q)$.
Then we will write $1=(\tau+1+k-j)^{-1}(\tau+1+k-j)$ and apply H\"older's inequality,
\begin{equation}\label{e.subs}
\begin{split}
&\sum_{Q\in\mathcal{W}(G)} \ell(Q)^{n-sp}\lvert f_{Q^*}\rvert^p\\
&\lesssim \sum_{k=-\infty}^\infty \sum_{Q\in \mathcal{W}_k(G)} \ell(Q)^{n-sp}
\bigg\{\sum_{j=-\infty}^{k+\tau} \sum_{\substack{R\in\mathcal{W}_j(G) \\ R\in \mathcal{C}(Q) }}
\fint_{R^*} \lvert f-f_{R^*}\rvert \bigg\}^p\\
&\lesssim \sum_{k=-\infty}^\infty \sum_{Q\in \mathcal{W}_k(G)} \ell(Q)^{n-sp}
\sum_{j=-\infty}^{k+\tau} (\tau+1+k-j)^p \bigg\{\sum_{\substack{R\in\mathcal{W}_j(G)\\ R\in \mathcal{C}(Q) }}
\fint_{R^*} \lvert f-f_{R^*}\rvert \bigg\}^p\,.
\end{split}
\end{equation}
By property (2) of chain $\mathcal{C}(Q)$ and H\"older's inequality, for any $j\in \Z$,
\begin{align*}
\sum_{\substack{R\in\mathcal{W}_j(G)\\ R\in  \mathcal{C}(Q) }}
\fint_{R^*} \lvert f-f_{R^*}\rvert
 &\lesssim \bigg\{ \sum_{\substack{R\in\mathcal{W}_j(G)\\ R\in \mathcal{C}(Q) }}
 \fint_{R^*} \lvert f-f_{R^*}\rvert^p\bigg\}^{1/p}\,.
\end{align*}
We substitute the last inequality to \eqref{e.subs}. Next,
we change the order of summation, and apply an equivalence for Whitney cubes:
$R\in\mathcal{C}(Q)$ if and only if $Q\in\mathcal{S}(R)$,
\begin{align*}
\sum_{Q\in\mathcal{W}(G)} \ell(Q)^{n-sp}\lvert f_{Q^*}\rvert^p
&\lesssim \sum_{k=-\infty}^\infty \sum_{Q\in \mathcal{W}_k(G)} \ell(Q)^{n-sp}
\sum_{j=-\infty}^{k+\tau} (\tau+1+k-j)^p  \sum_{\substack{R\in\mathcal{W}_j(G) \\ R\in  \mathcal{C}(Q) }}
 \fint_{R^*} \lvert f-f_{R^*}\rvert^p\\
 &= \sum_{j=-\infty}^\infty \sum_{R\in \mathcal{W}_j(G)} \ell(R)^{n-sp}  \fint_{R^*} \lvert f-f_{R^*}\rvert^p
\cdot \mathbf{A}_{j,R}\,.
\end{align*}
The constants
\[
\mathbf{A}_{j,R} = \sum_{k=j-\tau}^\infty \sum_{\substack{Q\in\mathcal{W}_k(G) \\ Q\in  \mathcal{S}(R) }} 
\bigg(\frac{\ell(Q)}{\ell(R)}\bigg)^{n-sp}(\tau+1+k-j)^p\,,\qquad j\in \Z\,,R\in\mathcal{W}_j(G)\,,
\]
are uniformly bounded in $j$ and $R$ by condition (3) in Proposition \ref{p.admissible}. 
Applying inequalities \eqref{e.first_basic} finishes the proof.
\end{proof}

\begin{remark}
The main line of the proof above is similar to the proof of the following
well known Hardy inequality for series.
If  $1\le p\le  \infty$ and $a_j\geq 0$, $j=0,1,\dots$, then
\begin{equation}\label{leindler}
\sum_{j=0}^\infty 2^{\sigma j}\bigg(\sum_{i=0}^j a_i\bigg)^p\le c\sum_{j=0}^\infty 2^{\sigma j}a_j^p \,\,\,\,\,\,\text{for}\,\,\sigma<0\,.
\end{equation}
See e.g. \cite{Leindler}
\end{remark}

\section{Proof of Theorem \ref{t.john_hardy_TL}}

In this section, we prove our main result.

\subsection*{Chain decomposition}
Suppose
$G$ is a bounded or unbounded domain in $\R^n$, $n\ge 2$, with a compact boundary
satisfying $\mathrm{dim}_{\mathcal{A}}(\partial G)<n-sp$,
where $1<p<\infty$ and $0<s<n/p$.
Let $\gamma=7\sqrt n$, and recall
 definition \eqref{c_definition} of $\mathcal{C}:=\mathcal{C}_{\partial G,\gamma}$.

By scaling and translating $G$, if necessary, we
may assume that there
is a dyadic cube $Q_0\in\mathcal{D}_0$ 
such that $Q\subset Q_0$ if 
\[Q\in\mathcal{W}^{G\textup{-small}}:=\{Q\in\mathcal{W}(G)\,:\, \ell(Q)\le \diam(\partial G)\}\,.
\]
In particular, this implies a relation $\partial G\subset Q_0$.
For a small Whitney cube $Q\in\mathcal{W}^{G\textup{-small}}$, we
let 
\[
\mathcal{C}(Q)=(Q_0,\ldots,Q_m)\subset \mathcal{C}\]
 be the unique chain of dyadic cubes
such that $Q_m=Q$
and $Q_{j-1}$ is the dyadic parent of $Q_{j}$, $j=1,\ldots,m$.
In particular,
\[Q\in\mathcal{D}_m\,.\]
The shadow of a  cube $R\in \mathcal{C}$ is
$\mathcal{S}(R) = \{Q\in\mathcal{W}^{G\textup{-small}}\,:\, R\in\mathcal{C}(Q)\}$.
Observe that
$\cup_{Q\in\mathcal{S}(R)} Q\subset R$ for all $R\in\mathcal{C}$.

\subsection*{Projection operators}
We also need certain projection operators.
For a cube $Q$ in $\R^n$ and $k\in \N_0$, we let
$P_{k,Q}$
 be a
projection from $L^1(Q)$ to $\mathcal{P}_{k-1}$ 
such that for every $1\le u\le \infty$ and every $f\in L^u(Q)$,
\begin{equation}\label{e.equiv}
\bigg(\fint_{Q}\lvert f(x)-P_{k,Q}f (x) \rvert^u\,dx\bigg)^{1/u}
\le C\mathcal{E}_k(f,Q)_{L^u(\R^n)}\,,
\end{equation}
where a constant $C$ depends on $ n$ and $k$.
For the  construction of these projection operators, we refer to
\cite[Proposition 3.4]{Shvartsman} and \cite{DeVoreSharpley}.

%

\begin{proposition}\label{chain0}
Suppose that $k\in \N_0$ and $Q\in\mathcal{W}^{G\textup{-small}}$. Then, for every $f\in L^1_{\textup{loc}}(\R^n)$,
\[
\lVert P_{k,Q} f-P_{k,Q_0} f \rVert_{L^\infty(Q)}\le C\sum_{R\in\mathcal{C}(Q)} \mathcal{E}_k(f,R)_{L^1(\R^n)}\,,
\]
where a constant $C$ depends on $n$ and $k$.
\end{proposition}

%
%

\begin{proof} 
Recall that $\mathcal{C}(Q)=(Q_0,\ldots,Q_m)$, with $Q_m=Q\in\mathcal{W}^{G\textup{-small}}$.
We claim that, for every $j\in \{1,\ldots,m\}$,
\begin{equation}\label{e.thisest}
\lVert P_{k,Q_j} f - P_{k,Q_{j-1}} f\rVert_{L^\infty(Q_j)}
\le C\mathcal{E}_{k}(f,Q_{j-1})_{L^1(\R^n)}\,,
\end{equation}
where $C$ depends on $n$ and $k$.
Indeed, 
let us first recall that $Q_j\subset Q_{j-1}$, and $2\ell(Q_j)=\ell(Q_{j-1})$.
By
a reverse H\"older inequality for polynomials
\cite[\S 3]{DeVoreSharpley}, and inequalities \eqref{e.equiv} and
\eqref{eqMonotonyOfLocalApproxSset},
\begin{multline*}
\lvert Q_j\rvert\lVert P_{k,Q_j} f - P_{k,Q_{j-1}} f\rVert_{L^\infty(Q_j)}
\lesssim \lVert P_{k,Q_j} f-P_{k,Q_{j-1}} f \rVert_{L^1(Q_j)}\\\le
\lVert f- P_{k,Q_j} f\rVert _{L^1(Q_j)}+\lVert f -P_{k,Q_{j-1}} f\rVert_{L^1(Q_{j-1})}
\lesssim \lvert Q_{j-1}\rvert \mathcal{E}_k(f,Q_{j-1})_{L^1(\R^n)}\,.
\end{multline*}
Thus, inequality \eqref{e.thisest} follows.

Since $Q\subset Q_j$
for every $j=0,\ldots,m$,
by inequality \eqref{e.thisest},
\begin{align*}
\lVert P_{k,Q} f-P_{k,Q_0} f\rVert_{L^\infty(Q)}
&\le \sum_{j=1}^{m}\lVert P_{k,Q_{j}} f-P_{k,Q_{j-1}} f\rVert_{L^\infty(Q_{j})}
\le C\sum_{j=0}^m \mathcal{E}_{k}(f,Q_j)_{L^1(\R^n)}.
\end{align*}
This concludes the proof of the proposition.
\end{proof}

\subsection*{Proof of Statement (A)}\label{s.proof_a}
We rewrite statement (A) as the following proposition.

\begin{proposition}
Let $n\ge 2$,  $1<p<\infty$, and $0<s<n/p$.
Suppose $G$ is a bounded or unbounded domain in $\R^n$,
with a compact boundary satisfying $\mathrm{dim}_{\mathcal{A}}(\partial G)<n-sp$.
Then 
\begin{equation}\label{e.hsd}
\bigg(\int_{G} \frac{\lvert f(x)\rvert^p}{\mathrm{dist}(x,\partial{G})^{sp}}\,dx\bigg)^{1/p}
\lesssim \lVert f \rVert_{F^s_{pq}(G)}
\end{equation}
 for all $f\in F^{s}_{pq}(G)$ and $1\le q<\infty$.
The implied
constant in \eqref{e.hsd} depends on n, $s$, $p$, and $\partial G$.
\end{proposition}

\begin{proof}
The embeddings $F^{s}_{pq'}(G)\subset F^{s}_{pq}(G)$ are trivially bounded if $q'\le q$. Hence,
it suffices to consider the case of $p\le q<\infty$. 
This allows us to set $u=p$ 
and $k=[s]+1>s$ in the definition of Triebel--Lizorkin space $F^s_{pq}(\R^n)$.

Let $f\in F^s_{pq}(G)$. Then, by definition, it suffices to show that
\begin{equation}\label{tama}
\int_{G} \frac{\lvert g(x)\rvert^p}{\dist(x,\partial G)^{sp}}\,dx\lesssim 
\lVert g\rVert_{F^s_{pq}(\R^n)}^p\,,
\end{equation}
where 
 $g\in F^{s}_{pq}(\R^n)$ is any extension of $f$, that is, $g\vert_G = f$ almost everywhere.
By first  inequality in  \eqref{dist_est}, we can bound
the left hand side of \eqref{tama} by a constant multiple of
\begin{equation}\label{e.core}
 \sum _{\substack{Q\in \mathcal{W}(G)}}
\ell(Q)^{n-sp} \fint_Q \lvert g(x)-P_{k,Q}g(x)\rvert^p + \lvert P_{k,Q}g(x)\rvert^p\,dx\,.
\end{equation}
By inequalities \eqref{e.equiv} and
\eqref{eqMonotonyOfLocalApproxSset},
\begin{equation*}
\begin{split}
&\sum _{\substack{Q\in \mathcal{W}(G)}}
\ell(Q)^{n-sp} \fint_Q \lvert g(x)-P_{k,Q} g(x)\rvert^p\,dx\\
&\lesssim     \lVert g\rVert_{L^p(\R^n)}^p+\sum _{\substack{Q\in \mathcal{W}(G) \\ 2\ell(Q)\le 1}}
\ell(Q)^{n-sp} \mathcal{E}_k(g,Q)_{L^p(\R^n)}^p \\
&\lesssim \lVert g\rVert_{L^p(\R^n)}^p+ \bigg\lVert \bigg\{ \sum_{\substack{Q\in\mathcal{W}(G)\\2\ell (Q)\le 1}} \chi_Q \ell(Q)^{-sq}
\mathcal{E}_k(g,Q)_{L^p(\R^n)}^q\bigg\}^{1/q}\bigg\rVert_p^p
 \lesssim 
 \lVert g\rVert_{F^s_{pq}(\R^n)}^p\,.
\end{split}
\end{equation*}
In the penultimate step, we used the fact that every point in $\R^n$ belongs
to at most $C=C(n)$ Whitney cubes. And, the last step  follows from monotonicity
\eqref{eqMonotonyOfLocalApproxSset} of the local approximation.

Then we estimate the remaining term in  \eqref{e.core}, i.e.,
\begin{equation}\label{e.rem}
\sum_{Q\in\mathcal{W}(G)}\ell(Q)^{n-sp}\fint_{Q} \lvert P_{k,Q} g(x)\rvert^p\,dx\,.
\end{equation}
This series, when restricted to big cubes
$Q\in\mathcal{W}(G)$ satisfying $\ell(Q)>\diam(\partial G)$, is bounded
by $C\lVert g\rVert_{L^p(\R^n)}^p$. Indeed,
this is an easy consequence of
inequality \eqref{e.equiv}.

Let us estimate  the remaining part of series \eqref{e.rem}, where
the summation is restricted to small cubes $\mathcal{W}^{G\textup{-small}}$.
In order to do this, we write $P_{k,Q} g = P_{k,Q_0}g +(P_{k,Q}g - P_{k,Q_0} g)$,
and estimate the resulting two series, denoted by $S_1$ and $S_2$.
First, by a reverse H\"older inequality for polynomials and the
assumption $\mathrm{dim}_{\mathcal{A}}(\partial G) < n-sp$,
\begin{align*}
S_1:=\sum_{Q\in\mathcal{W}^{G\textup{-small}}} \ell(Q)^{n-sp} &\fint_Q \lvert P_{k,Q_0}g(x)\rvert^p\,dx 
\le \lVert P_{k,Q_0} g\rVert_{L^\infty(Q_0)}^p \sum_{Q\in\mathcal{W}^{G\textup{-small}}}
\ell(Q)^{n-sp}\\
&\lesssim \lVert P_{k,Q_0} g\rVert_{L^p(Q_0)}^p  \int_{Q_0}\dist(x,\partial G)^{-sp}\,dx
\lesssim \lVert g\rVert_{L^p(\R^n)}^p\,.
\end{align*}
By Proposition \ref{chain0},
\begin{equation*}
\begin{split}
S_2:&=\sum_{Q\in\mathcal{W}^{G\textup{-small}}}\ell(Q)^{n-sp} \fint_Q \lvert P_{k,Q}g(x) - P_{k,Q_0}g(x)\rvert^p\,dx \\
&\lesssim  \sum_{m=0}^\infty \sum_{Q\in \mathcal{W}_m^{G\textup{-small}}} \ell(Q)^{n-sp} 
\bigg\{ \sum_{j=0}^m (1+m-j)^{-1}(1+m-j)
\fint_{R_j^Q} \lvert g(x)-P_{k,R_j^Q} g(x)\rvert\,dx \bigg\}^p
\end{split}
\end{equation*}
where we use  notation ${R}^Q_j$ for the unique cube $R\in \mathcal{C}(Q)\cap\mathcal{D}_j$.
Next, proceeding as in the proof of Theorem \ref{t.john_hardy}, 
we obtain
\begin{equation}\label{e.tag}
\begin{split}
S_2 &\lesssim \sum_{j=0}^\infty \sum_{R\in \mathcal{C}\cap\mathcal{D}_j} \ell(R)^{n-sp} \fint_{R} \lvert g(x)-P_{k,R} g(x)\rvert^p\,dx
\cdot \mathbf{A}'_{j,R}
\\&\lesssim \sum_{R\in\mathcal{C}} \ell(R)^{n-sp} 
\fint_{R} \lvert g(x)-P_{k,R} g(x)\rvert^p\,dx\,.
\end{split}
\end{equation}
Here the uniform boundedness in $j$ and $R$ of the constants $\mathbf{A}'_{j,R}$
can be easily shown as in the proof of Proposition \ref{p.admissible}.

By 
Inequality \eqref{e.equiv}, 
Remark \ref{porous_obs}, and a reverse H\"older inequality
 in Theorem \ref{reverse}, we can bound term $S_2$ by
a constant multiple of
\begin{equation*}
\begin{split}
\bigg\lVert \bigg\{ \sum_{R\in\mathcal{C}} \chi_R \ell(R)^{-sp}
\mathcal{E}_k(g,R)_{L^p(\R^n)}^p\bigg\}^{1/p}\bigg\rVert_p^p
\lesssim  \bigg\lVert \bigg\{ \sum_{R\in\mathcal{C}} \chi_R \ell(R)^{-sq}
\mathcal{E}_k(g,R)_{L^p(\R^n)}^q\bigg\}^{1/q}\bigg\rVert_p^p\,,
\end{split}
\end{equation*}
where $\mathcal{C}=\mathcal{C}_{\partial G,\gamma}$.
The observation that the last term is dominated
by the required upper bound $C\lVert g\rVert_{F^s_{pq}(\R^n)}^p$ finishes the proof.
\end{proof}

\subsection*{Proof of Statement (B)}\label{s.proof_b}
This statement is covered by the following proposition.

\begin{proposition}\label{t.hardy_admissible_general}
Let $n\ge 2$,  $1\le p<\infty$, and $0<s<n/p$.
Suppose that $G$ is domain in $\R^n$, with a compact boundary,
such that $\R^n\setminus G$ has zero Lebesgue measure,
and inequality
\begin{equation}\label{e.Hardy}
\bigg(\int_{G} \frac{\lvert f(x)\rvert^p}{\mathrm{dist}(x,\partial{G})^{sp}}\,dx\bigg)^{1/p}
\lesssim \lVert f \rVert_{F^s_{pq}(G)}
\end{equation}
holds for some $1\le q\le \infty$ and
 for all $f\in F^{s}_{pq}(G)$.
 Then $\mathrm{dim}_{\mathcal{A}}(\partial G)< n-sp$.
\end{proposition}

\begin{proof}
We rely on the following homogeneity property, \cite[Corollary 5.16]{T3}. Namely,
\begin{equation}\label{hom}
\lVert f(r\cdot)\rVert_{F^{s}_{pq}(\R^n)}
\simeq r^{s-n/p} \lVert f\rVert_{F^s_{pq}(\R^n)}
\end{equation}
for every $0<r\le 1$ and every $f\in F^{s}_{pq}(\R^n)$ supported
in $B(0,r)=\{x\in \R^n\,:\, \lvert x\rvert<r\}$.

Let us consider a point $x\in\partial G$ and a radius $0<r\le 1$. Without loss of generality, we may assume
that $x=0$.
Fix a function $\varphi\in C^\infty_0(\R^n)$,
supported in $B(0,1)$, and satisfying  $\varphi(x)=1$ if $x\in B(0,1/2)$.
Denote $f(y)=\varphi(y/r)$ for $y\in \R^n$. 
Since the measure of $\R^n\setminus G$ is zero, 
\begin{align}\label{AikawaEstimate}
\int_{B(x,r/2)} \dist(y,\partial G)^{-sp}\,dy &\le 
\int_{G} \frac{\lvert f(y)\rvert^p}{\dist(y,\partial G)^{sp}}\,dy\
\lesssim \lVert f\vert_G\rVert_{F^{s}_{pq}(G)}^p\\&\le \lVert f\rVert_{F^{s}_{pq}(\R^n)}^p \simeq r^{n-sp} \lVert f(r\cdot)\rVert_{F^s_{pq}(\R^n)}^p\lesssim r^{n-sp}\,.
\end{align}
By inequality \eqref{AikawaEstimate} and Remark
\ref{r.compact}, we have $n-sp\in\mathcal{A}(\partial G)$. To show that the Aikawa dimension is, indeed,  strictly less than $n-sp$,
we proceed as in the proof of \cite[Lemma 2.4]{KoskelaZhong}.
Since $n-sp\in\mathcal{A}(\partial G)$ it is
 straightforward to verify that,
for every $x\in \R^n$ and $r>0$,
\begin{equation}\label{ReverseHolder}
\bigg(\fint_{B(x,r)} \dist(y,\partial G)^{-sp}\,dy\bigg)^{1/p} \lesssim \fint_{B(x,r)} \dist(y,\partial G)^{-s}\,dy\,.
\end{equation}
Hence, by the self-improving properties of reverse H\"older inequalities, \cite[Lemma 3]{gehring}, we find that 
$n-sp-\delta\in\mathcal{A}(\partial G)$ for some $\delta>0$, and the claim follows.
\end{proof}

\section{Applications}\label{s.application}



In this section, we study two problems closely related to Hardy-inequalities: the boundedness of the zero extension
operator and the boundedness of pointwise multiplier operators.
First, let us consider the zero extension
operator $E_0:W^{s,p}(G)\to W^{s,p}(\R^n)$, recall definition \eqref{zero_ext_def}. 
\begin{lemma}\label{l.chi_bounded}
Suppose that $G$ is a proper domain in $\R^n$. Let $0<s<1$ and $1<p<\infty$.
Then the zero extension
$E_0f$
 of any  $f\in W^{s,p}(G)$ satisfies the following inequality
\[
\lVert E_0f \rVert_{W^{s,p}(\R^n)}
\lesssim \lVert f\rVert_{W^{s,p}({G})} + \bigg(\int_{{G}} 
\frac{\lvert f(x)\rvert^p}{ \mathrm{dist}(x,\partial G)^{sp}}\,dx\bigg)^{1/p}\,,
\]
where the implied constant depends on $n$, $s$, and $p$.
\end{lemma}

\begin{proof}
By definition \eqref{d.frac_def},
 $\lVert E_0f\rVert_{L^p(\R^n)}=\lVert f\rVert_{L^p(G)}\le \lVert f\rVert_{W^{s,p}(G)}$.
Next, let us write
\begin{align*}
\lvert E_0f \rvert_{W^{s,p}(\R^n)}^p &=
\int_{\R^n}\int_{\R^n}\frac{\lvert E_0f(x) - E_0f(y)\rvert^p}
{\lvert x-y\rvert^{n+sp}}\,dy\,dx\\
&\le \lvert f\rvert_{W^{s,p}(G)}^p + 2 \int_{{G}} 
\lvert f(x)\rvert^p \int_{\R^n\setminus {G}} \frac{1}{\lvert x-y\rvert^{n+sp}}\,dy\,dx\,.
\end{align*}
It remains to apply the following estimates, which are valid
for $x\in{G}$,
\begin{align*}
&\int_{\R^n\setminus {G}} \frac{1}{\lvert x-y\rvert^{n+sp}}\,dy\\
&\le \int_{\R^n\setminus B(x,\mathrm{dist}(x,\partial{G}))} \frac{1}{\lvert x-y\rvert^{n+sp}}\,dy
\lesssim \int_{\mathrm{dist}(x,\partial G)}^\infty r^{-1-sp}\,dr \lesssim 
\mathrm{dist}(x,\partial{G})^{-sp}\,.
\end{align*}
This completes the proof of the lemma.
\end{proof}
The following theorem is a consequence of Lemma \ref{l.chi_bounded}
and Theorem \ref{t.john_hardy}.
\begin{theorem}\label{t.equivalences}
Suppose that $n\ge 2$, $1<p<\infty$, and $0<s<\min\{1,n/p\}$. Let $G$ be
a John domain in $\R^n$ such that
$\mathrm{dim}_{\mathcal{A}} (\partial G)< n-sp$. Then, for every $f\in W^{s,p}(G)$,
\begin{equation}\label{e.equivalences}
\lVert E_0 f\rVert_{W^{s,p}(\R^n)}\le C \lVert f\rVert_{W^{s,p}(G)}\,,
\end{equation}
where a constant $C$ depends on $n$, $s$, $p$, and $G$.
\end{theorem}
The assumption that $G$ is a John
domain can be relaxed, e.g., by Theorem \ref{t.john_hardy_TL}
and extension results in \cite{Z}. 
However, the proof of Theorem \ref{t.equivalences} under the John
assumption has the advantage of being rather simple.

Let us now study the boundedness of pointwise multiplier operators. 
The following proposition is proved in \cite[Proposition 4.1]{ihnatsyeva2}.
\begin{proposition}\label{ext_hardy}
Let $G$ be a domain in $\R^n$ whose boundary is porous in $\R^n$.
 Let $f\in F^{s}_{pq}(\R^n)$, $1<p<\infty$, $1\le q< \infty$, and $s>0$.
Then 
\begin{equation}
\lVert f\chi_G \rVert_{F^{s}_{pq}(\R^n)}\lesssim \lVert f\rVert_{F^{s}_{pq}(\R^n)}+\bigg(\int_G\frac{|f(x)|^p}{\dist(x,\partial G)^{s p}}\,dx\bigg)^{1/p}.
\end{equation}
The implied constant depends on $p$, $q$, $s$, $n$, and $\partial G$.
\end{proposition}
%

\begin{theorem}\label{t.equivalences_TL}
Let $n\ge 2$, $1<p<\infty$, and $0<s<n/p$.
Suppose that  $G$ is a domain in $\R^n$, with a compact boundary, such that
$\mathrm{dim}_{\mathcal{A}}(\partial G)< n-sp$. Then,
for every $1\le q< \infty$, the pointwise multiplier operator
$f\mapsto \chi_G f$
is bounded on $F^{s}_{pq}(\R^n)$.
\end{theorem}
\begin{proof}
Observe that $\partial G$ is porous in $\R^n$, see Remark \ref{porous_obs}.
By Proposition \ref{ext_hardy} and Theorem \ref{t.john_hardy_TL}, for
$f\in F^{s}_{pq}(\R^n)$,
\begin{align*}
\lVert f\chi_G\rVert_{F^{s}_{pq}(\R^n)}
\lesssim \lVert f\rVert_{F^{s}_{pq}(\R^n)} +\lVert f\vert_G\rVert_{F^{s}_{pq}(G)}
\le 2\lVert f\rVert_{F^{s}_{pq}(\R^n)}\,.
\end{align*}
This concludes the proof.
\end{proof}

The statement of Theorem \ref{t.equivalences_TL} is not new and it
is covered  by \cite[Theorem 13.3]{FJ}, whose
proof is more technical, based upon atomic decompositions.
We also refer to \cite[\S 4.1]{sickel}.

\end{document}